\documentclass[12pt,leqno,amscd, amsfonts, amssymb,pstricks,verbatim]{amsart}
\usepackage{amsfonts}
\usepackage{mathrsfs}
\usepackage{amsmath}
\usepackage{amsmath,graphics}
\usepackage{fancyhdr}
\usepackage{amssymb}
\usepackage{indentfirst,latexsym,bm,amsmath,pstricks,amssymb,amsthm,pxfonts,graphicx}
\usepackage[all]{xy}
\usepackage{amsthm}
\usepackage{times}
\usepackage{datetime}
\usepackage{scrtime}
\usepackage{amscd}
\usepackage[CJKbookmarks=true]{hyperref}
\def\@evenhead{\thepage \hfill \shortauthor \hfill}
\def\@oddfoot{}
\def\@evenfoot{}

\def\@evenhead{\thepage \hfill \shortauthor \hfill}
\def\@oddfoot{}
\def\@evenfoot{}
\newtheorem{thm}{Theorem}[section]
\newtheorem{lem}[thm]{Lemma}
\newtheorem{cor}[thm]{Corollary}
\newtheorem{prop}[thm]{Proposition}

\newtheorem{example}[thm]{Example}

\newtheorem{rem}[thm]{Remark}
\newtheorem{que}[thm]{Question}

\numberwithin{equation}{section}

\newcommand{\p}{\partial}
\def\N{\mathcal{N}}

\def\D{\mathcal D}

\def\O{\mathcal O}

\def\g{\mathfrak{g}}
\def\d{\partial}

\def\ker{{\rm Ker\,}}
\def\im{{\rm Im\,}}
\def\Aut {{\rm Aut\,}}

\def\ad{{\rm ad}}
\def\Lie{{\rm Lie}}

\def\m{\mathfrak{m}}
\def\A {\mathbb{A}^n}
\def\v {\varphi}
\def\e {\eta}
\def\wt{\widetilde}
\def\div{{\rm div\,}} 
\begin{document}
\author{Hao Chang}
\title[$W_n$ and $S_n$]{Adjoint quotient maps for restricted Lie algebras $W_n$ and $S_n$}

\address{Mathematisches Seminar, Christian-Albrechts-Universit\"{a}t zu Kiel, 24098 Kiel, Germany.}
\email{chang@math.uni-kiel.de}
\thanks{This work which is part of the author's thesis is supported by China Scholarship Council (210306140109)}

\begin{abstract}
We give an explicit description of adjoint quotient maps for Jacobson-Witt algebra $W_n$ and special algebra $S_n$. An analogue of Kostant's differential criterion of regularity is given for $W_n$. Furthermore, we describe the fiber of adjoint quotient map for $S_n$ and construct the analogs of Kostant's transverse slice. In addition, we analysis the nilpotent elements and generalize some Premet's results to $S_n$.
\quad \\[0pt]
\end{abstract}

\maketitle

\section{Introduction}
In \cite{Ko}, Kostant established a series of results about the adjoint quotient maps of the classical Lie algebras.
For the Lie algebras of Cartan type, Premet investigated the nilpotent elements of restricted Lie algebra $W_n$, obtained the analogues of results of Chevalley and Kostant which are well known in the classical theory of Lie algebras. Recently, Bois, Farnsteiner and Shu \cite{BFS} generalized Premet's results to other Lie algebras of Cartan type by defining the Weyl group for non-clasical restricted Lie algebras. Later, in \cite{WCL}, the authors studied the nilpotent variety for the special algebra $S_n$ and proved the nilpotent variety of $S_n$ is an irreducible affine algebraic variety. This indeed confirmed a conjecture of Premet (cf. \cite{P1},\cite{Sk},\cite{WCL}).

Let $\g$ be a restricted Lie algebra and $G$ its automorphism group. Based on the above results, the ring of invariants $k[\g]^G$ was completely described for the cases $W_n$ and $S_n$. Now let $\Phi_G:\g\rightarrow \g/G$, we call $\Phi_G$ the adjoint quotient map. According to \cite{Pr-1} and \cite{WCL}, the ring of invariants $k[\g]^G$ is polynomial ring for $W_n$ and $S_n$. Let $P_{\xi}$ denote the fiber of $\Phi_G^{-1}(\xi)$, where $\xi\in \g/G$. In \cite{Pr-1}, Premet showed that $\Phi_{\xi}$ is irreducible complete intersection for $\g=W_n$. The aim of the paper is to establish the analogue of Kostant differential criterion of regularity for $W_n$ and extend some Premet's results to $S_n$.

Our paper is organized as follows. After recalling some basic definitions and results on the Jacobson-Witt algebra $W_n$ and special algebra $S_n$, we explicitly describe the fiber $P_{\xi}$ for $W_n$. We give a criterion to determine the smooth locus of $P_{\xi}$ (Proposition \ref{pro-1}). In Section \ref{special algebraSn}, we continue to investigate the fiber of adjoint quotient map for $S_n$. The fiber is shown to be an irreducible complete intersection. Motivated by the results of \cite{WCL}, we construct the analogs of Kostant's transverse slice for $W_n$ and $S_n$ and give a description of nilpotent elements in $S_n$.

Some results of Section \ref{Jacobson-Witt algebraWn} are also proved by Premet in his recent preprint \cite{Pr-2} using slightly different methods.
\subsection*{Acknowledgements}
I would like to thank Jean-Marie Bois, Rolf Farnsteiner and Alexander Premet for useful discussions on an earlier version of this paper.
\section{Preliminaries and notations}

\subsection{Restricted Lie algebras of Cartan type $W$}\label{1Wn}
Let $k$ be an
algebraically closed field of characteristic $p>3$,
and let $B_n$ be the truncated polynomial ring in $n$ variants:
$$B_n=k[x_1,\cdots,x_n]/(x_1^p,\cdots,x_n^p)\,\,(n\geq 3).$$

Let $W_n\coloneqq {\rm Der}(B_n)$ be the Jacobson-Witt algebra.
It is a simple restricted Lie algebra and its
$p$-map is the standard $p$-th power of linear operators. Denote by
$D_{i}=\p/\p x_i\in W_n$ the partial derivative with respect to the
variable $x_i$. Then $\{D_1,\cdots,D_n\}$ is a basis of the
$B_n$-module $W_n$, so that $\dim_{k}W_n=np^n$.

Let $0 \neq f\in B_n$.
Throughout the article,
\begin{equation}
\text{ $\deg f$ is the {\bf minimal} degree of each monomial in $f$.} \label{d}
\end{equation}
And let $\deg 0= +\infty $ conventionally.
There is a natural filtration in $B_n$:
$$(B_n)_i=\left\{f \in B_n\,|\,\deg f \geq i\right\},
\qquad \text{for $i=0,\,\cdots,\,n(p-1)$}.$$
Then $W_n$ inherits a filtration,
$$(W_n)_i=\left\{\sum\limits_{j=1}^n  f_jD_j\in W_n \, \big|\,
\deg f_j \geq i+1, \forall j\right\}, \qquad
\text{for $i=-1,\,\cdots,\,n(p-1)-1$}.$$

\subsection{Restricted Lie algebras of Cartan type $S$}\label{1Sn}
The Lie algebra $\wt S_n$ is defined as
$$\wt S_n=\left\{ \sum\limits_{i=1}^nf_iD_i \in W_n \,\big|\,
\div f=\sum\limits_{i=1}^n D_i f_i=0\right\}.$$
The special algebra $S_n$ is the derived
algebra of $\widetilde{S}_n$.
As a vector space, $S_n$ is spanned by the derivations
\begin{equation}\label{S_n generator}
D_{i,j}\{u\}=D_j(u)D_i-D_i(u)D_j\,, \quad u \in B_n\,.
\end{equation}

The Lie algebras $\wt S_n$ and $S_n$ are both restricted Lie algebras
with the same $p$-map as the Lie algebra $W_n$. Moreover,
it is well-known that $S_n$ is a simple restricted Lie algebra with
$\dim_{k}S_n=(n-1)(p^n-1)$ and
$\dim_{k} \wt{S}_n=(n-1)p^n+1$ (cf. \cite[Chapter 4, Theorem 3.5 and 3.7]{SF}).
And there are $n+1$ (resp. $n$) classes of maximal tori of dimension $n$
(resp. $n-1$) in $W_n$ (resp. $S_n$) (cf. \cite{D}).

\subsection{Automorphism groups of $W$ and $S$}\label{1automorphismgroup}
Let $G=\Aut(W_n)$ and $\wt G=\Aut (S_n)$ be the restricted
automorphism groups of  $W_n$ and $S_n$ respectively.
It is a classical result (cf. \cite{D}) that
any automorphism $\phi$ of the $k$-algebra $B_n$ induces an automorphism of
$W_n$ by
\begin{equation}\label{aut}
g(D)=\phi\circ D \circ \phi^{-1}, \quad \text{for any~} D \in W_n\,,
\end{equation}
and in this way,
\begin{eqnarray}
\Aut(W_n) &\cong& \Aut(B_n), \nonumber\\
\Aut(S_n) &\cong& \left\{ \phi \in \Aut(B_n)\,\big|\,
\det\big(\p_i\phi(x_j)\big) \text{~is a non-zero constant}\right\}. \label{aut1}
\end{eqnarray}
According to \cite[Theorem 2]{W},
$G$ and $\wt G$ are  both connected
algebraic groups.
Furthermore, let $\mathfrak{m}$ be the unique maximal ideal of $B_n$.
$G$ is of dimension $np^n-n$ and $\Lie(G)=(W_n)_{0}=\{\sum f_{i}\d_{i}~|~f_i\in\mathfrak{m}\}$(cf.\cite{LN}),
and $\dim \wt G=\dim_{k} S_n+1$ (\cite[Proposition 3.4]{WCL}).
\vspace{0.1cm}

\section{Jacobson-Witt algebra $W_n$}\label{Jacobson-Witt algebraWn}

From now on, we always denote by $\g$ the Jacobson-Witt algebra $W_n$.
\subsection{smooth points}
For any $x\in \g$, we can consider the subring of $x$-constants $B_n^x\subseteq B_n$
to be the space of truncated polynomials annihilated by $x$. We always have $k\subseteq B_n^x.$
Let $\rho$ be the tautological representation of $\g$ in the space $B_n$. Let $P(t,x)$ be the characteristic polynomial of $\rho(x).$
By\cite[Corollary 1]{Pr-1},
\begin{equation}\label{characteristicpolynomialofWn}
P(t,x)=t^{p^n}-\sum_{i=0}^{n-1}\phi_i(x)t^{p^i}.
\end{equation}
Let $\g^x$ be the centraliser of $x$ in $\g$. There exists an open subset of $\g$ consisting of elements $x$ such $\dim \g^x=n$ (\cite[Lemma 1]{Pr-1}).
Hence, $n={\rm min}\{\dim \g^x ~|~x\in \g \}$.

Consider the following open subsets of $\g$
$$U_1=\{x\in \g~|~B_n^x=k\},$$
$$U_2=\{x\in \g~|~\dim \g^x=n\},$$
$$U_3=\{x\in \g~|~(d\phi_0)_x,(d\phi_1)_x,\cdots,(d\phi_{n-1})_x ~\rm{are~~linearly~~independent}\}.$$

Set $$M_{\ad x}=(\ad x)^{p^n-1}-\sum_{i=0}^{n-1}\phi_i(x)(\ad x)^{p^i-1},$$
and
$$M_x=\rho(x)^{p^n-1}-\sum_{i=0}^{n-1}\phi_i(x)\rho (x)^{p^i-1}.$$

We will frequently use the following  two lemmas:
\begin{lem}(\cite[Lemma 7.(i)]{Pr-1})\label{for-1}
For any $x,y\in \g$,
$$M_{\ad x}(y)=\sum_{i=0}^{n-1}(d\phi_i)_x(y).x^{p^i}$$
\end{lem}

\begin{lem}\label{minimalpolynomialischaracteristicpolynomial}
Let $x\in W_n$. If $B_n^x=k$, then the minimal polynomial of $\rho(x)$ is just
the polynomial $P(t; x)$.
\end{lem}
\begin{proof}
See the proof of Proposition 1.2.3 in \cite{Bois}.
\end{proof}

We have following proposition to describe the elements which have no non-trivial constants.
\begin{prop}\label{pro-1}Keep notations as above,
$U_1=U_3$.
\end{prop}

\begin{proof}
The nilpotent case is a direct corollary of \cite[Lemma 7]{Pr-1}. We always assume $x$ is neither $p$-nilpotent nor $p$-semisimple.
Let $x=x_s+x_n$ be the Jordan-Chevalley decomposition of $x$. Furthermore, we assume $x_s$ is not regular, for this case, it is easy.

If $B_n^x=k$, thanks to Lemma \ref{minimalpolynomialischaracteristicpolynomial}, the minimal polynomial of $\rho(x)$ is just the polynomial $P(t,x)$. It follows that $x,x^p,\cdots,x^{p^{n-1}}$ are linearly independent.
Thanks to Lemma \ref{for-1}, to verify $x\in U_3$ it is sufficient to check
$\dim M_{\ad x}(\g)=n$.
 For any $f\in B_n$ and $l\in \{0,\cdots,n-1\}$ we have
$$M_{\ad x}(f.x^{p^l})=M_x(f)x^{p^l}.$$
Since $M_x(B_n)\subseteq \ker\rho(x)=k$ and $M_x\neq 0$, we have
$$M_{\ad x}(\g)=k-\langle x,x^p,\cdots,x^{p^{n-1}}\rangle,$$
as desired.

On the other hand, if $x\in U_3$, then it can be easily verified that
$x,x^p,\cdots,x^{p^{n-1}}$ are linearly independent.
Furthermore, this implies that
$$\im M_{\ad x}=k-\langle x,x^p,\cdots,x^{p^{n-1}}\rangle.$$
We divide the discussion into two cases.

\textbf{Case 1:} $x\notin\g_{0}.$

Let $f$ be a nonzero element in $\m\cap B_n^x$. We may assume that ${\rm deg}(f)\geq 2$ (if not, we replace $f$ by $f^2\neq0$). By the choice of $f$ one has $fx\neq 0$ and $fx\in \g_{1}\cap\g^x$, so
$fx=a_{0}x+a_{1}x^p+\cdots+a_{n-1}x^{p^{n-1}}$.
Furthermore, $(fx)^p=f^px^p+(fx)^{p-1}(f)x$. Hence,
$$0=(fx)^p=(a_0)^px^p+(a_{1})^px^{p^2}+\cdots+(a_{n-1})^px^{p^{n}},$$
and
$$x^{p^{n}}=-(\frac{a_0}{a_{n-1}})^px^p-(\frac{a_1}{a_{n-1}})^px^{p^2}-\cdots-(\frac{a_{n-2}}{a_{n-1}})^px^{p^{n-1}}.$$
Since minimal polynomial of $\ad x$ is $P(t,x)$ , we have
$\phi_0(x)=0, \phi_{i}(x)=-(\frac{a_{i-1}}{a_{n-1}})^p,$ for $1\leq i\leq (n-1)$.
Let $a_{s}\neq 0$ and $a_l=0$ for $l< s$. There is a $y\in\g$ with $M_{\ad x}(y)=x$. Then
\begin{eqnarray*}
x&=&(\ad x)^{p^n-1}-\sum_{i=s+1}^{n-1}\phi(x)(\ad x)^{p^i-1}(y)\\
&=&[x^{p^s},(\ad x)^{p^n-p^s-1}-\sum_{i=s+1}^{n-1}\phi(x)(\ad x)^{p^i-p^s-1}(y)]\\
&=&\frac{1}{a_{s}}[fx-\sum_{i\geq s+1}a_{i}x^{p^i},(\ad x)^{p^n-p^s-1}-\sum_{i=s+1}^{n-1}\phi(x)(\ad x)^{p^i-p^s-1}(y)]\\
&=&\frac{1}{a_{s}}[fx,(\ad x)^{p^n-p^s-1}-\sum_{i=s+1}^{n-1}\phi(x)(\ad x)^{p^i-p^s-1}(y)]\in\g_0.
\end{eqnarray*}
We have obtained a contradiction.

\textbf{Case 2:} $x\in\g_{0}.$

 Since $\g_{0}$ is an algebraic Lie algebra, by a conjugation, we can write $x^{p^n}=\sum_{i=1}^{n}\lambda _ix_iD_i$. Note that $x_s$ is not regular, so we can find an element $g\in\m$ and $\deg g\geq2$ such that $x^{p^n}(g)=0$ but $gx^{p^n}\neq 0$.
 Since $[x,gx^{p^n}]=x(g)x^{p^n}$, there exists $g'\in\m$ such that $g'x^{p^n}\in\g^x$ and $x^{p^n}(g')=0$. $g'\in\im\rho(x)$ and $x\in \g_{0}$, it follows that
 $g'x^{p^n}\in\g_{2}$. Now one may use the same argument in Case 1.
 \end{proof}

\begin{rem}
1).~From the proof of Proposition \ref{pro-1}, we have
$$\{x\in \g~|~B_n^x=k\}=\{x\in \g~|~\dim M_{\ad x}(\g)=n\}.$$
This results turn out to be an analogues of results of Premet which describe the smooth points in nilpotent variety of $W_n$ (\cite[Lemma 7]{Pr-1}).\\
2).~In \cite[Corollary 1.2.4]{Bois}, the author proved that $U_1\subseteq U_2$.\\
3).~ In \cite[Theorem 1 and Theorem 2]{Pr-2}, Premet proved $U_1=U_2=U_3$.
\end{rem}

\subsection{adjoint quotient for $W_n$}\label{propertiesofvg}
Choose generators $\phi_{0},\phi_{1},\cdots,\phi_{n-1}\in k[\g]^G$ for $k[\g]^G$ as a $k$-algebra. Set
$$\Phi_G:\g\longrightarrow \mathbb{A}^n,~~~~~~~~~~~x\mapsto (\phi_{0}(x),\phi_{1}(x),\cdots,\phi_{n-1}(x)).$$
We call $\Phi_G$ the adjoint quotient since the comorphism of $\Phi_G$ is just the inclusion of $k[\g]^G$ into $k[\g]$.

Let $P_{\xi}$ denote the fiber $\Phi_{G}^{-1}(\xi)$, where $\xi=(\xi_{0},\cdots,\xi_{n-1})\in\mathbb{A}^n$.
According to \cite[Theorem 3]{Pr-1}, $\varphi_G$ is equidimensional and surjective, each of its fiber is irreducible, has dimension $N-n$. Furthermore, for any $\xi\in\A$, the fiber $P_{\xi}$ contains a point $\D_{\xi}$ with trivial stabilizer and the ideal
$I(P_{\xi})=\{f\in k[\g]~|~f(P_{\xi})=0\}$ is generated by polynomials $\phi_0(x)-\xi_{0},\cdots,\phi_{n-1}(x)-\xi_{n-1}$.
Thanks to \cite[Lemma 12]{Pr-1}, the elements of the orbit $\O_{\xi}=G.\D_{\xi}$ are smooth in $P_{\xi}$. Compare with the classical case (cf.\cite[Proposition 7.13]{Jan}), our situation is different. In general, the smooth locus of $P_{\xi}$ is not the unique dense open orbit $\O_{\xi}$.

\begin{que}
When the fiber $P_{\xi}$ is smooth?
\end{que}
The following Proposition answer the above question. It can also be found in \cite[Theorem 3(ii)]{Pr-2}.
\begin{prop}\label{pro-2}
Let $\xi=(\xi_{0},\cdots,\xi_{n-1})\in\mathbb{A}^n$. $P_{\xi}$ is smooth if and only if $\xi_{0}\neq 0.$
\end{prop}
\begin{proof}
If $P_{\xi}$ is smooth, then it follows from \cite[Lemma 7]{Pr-1} that $\xi\neq0$, i.e., $x$ is not nilpotent for any $x\in P_{\xi}$.
Take $y\in P_{\xi}$, by virtue of \cite[Lemma 14]{Pr-1}, $\overline{G.y}~\cap~T_n\neq \varnothing$, where $T_n=<x_1\d_1,\cdots,x_n\d_n>$. Let $0\neq y'\in \overline{G.y}~\cap~T_n$. Note that $\Phi_G(y')=\Phi_G(y)$ and $y'$ is also a smooth point in $P_{\xi}$. Hence $(d\phi_0)_{y'},(d\phi_1)_{y'},\cdots,(d\phi_{n-1})_{y'}$ are linearly independent. According to Proposition \ref{pro-1}, $B_n^{y'}=k$, so the minimal polynomial of $\rho(y')$ is just $P(t,y')$. Since, in addition, $y'$ is semisimple. Hence $y'$ is regular semisimple element. This implies
$\xi_{0}=\phi_{0}(y')\neq 0$.

On the other hand, if $\xi_{0}\neq 0$, for any $x\in P_{\xi}$, then $\phi_0(x)=\xi_{0}\neq 0$. It follows that $x$ is regular semisimple element. Hence $B_n^x=k$, Proposition \ref{pro-1} implies that $x$ is smooth in $P_{\xi}$.
\end{proof}

The following example was communicated to the author by Premet.
\begin{example}
For the Witt algebra $W_1$, consider the fiber $P_{1}$, i.e., the solution of $t^p-t=0$ in $W_1$. We have finitely many orbits in the fiber whose representatives are  $(1+x)\d$, $\lambda x\d$ with $\lambda\in \mathbb{F}_p^*$. In particular, the fiber is smooth.
\end{example}

\begin{cor}
If the smooth locus of $P_{\xi}$ equals $\O_{\xi}$, then $\xi_{0}=0$.
\end{cor}
\begin{proof}
It is sufficient to show that the fiber $P_{\xi}$ contains at least two orbits.
Suppose that $P_{\xi}$ consist of one orbit, i.e., $P_{\xi}=\O_{\xi}=G.\D_{\xi}$. By the same argument as in Proposition \ref{pro-2}, we can find an element $x\in T_n$ such that $P_{\xi}=G.x$. Note that $x\in\g_{0}$, this implies that $x,x^p,\cdots,x^{p^{n-1}}$ contain in $\g_{0}$. On the other hand, $x,x^p,\cdots,x^{p^{n-1}}$ form a basis of the $B_n$-module $\g$(\cite[Lemma 12]{Pr-1}). This is a contradiction.
\end{proof}
\begin{que}
When the smooth locus of $P_{\xi}$ equals $\O_{\xi}$ ?
\end{que}

\begin{rem}
The smooth locus of nilpotent variety $\N=P_{0}$ equals the unique dense open orbit $\O_{0}$ (\cite[Lemma 7]{Pr-1}).
\end{rem}

\section{Special algebra $S_n$}\label{special algebraSn}

\subsection{adjoint quotient for $S_n$}\label{propertiesofvgforSn}
Recall that $\wt G$ is the restricted automorphism group of $S_n$.
Furthermore, $\dim \wt G=\dim_{k} S_n+1=N+1$,
where, for convenience, let
$$N=(n-1)(p^n-1)=\dim S_n\,.$$

Recall that the characteristic polynomial $P(t,x)$ is a $p$-polynomial on $W_n$.
Let $\varphi_i$ be the restriction of $\phi_i$ on $S_n$.
Denote
\begin{equation}\label{characteristicpolynomialofSn}
Q(t,x)=t^{p^n}-\sum_{i=0}^{n-1}\v_i(x)t^{p^i}.
\end{equation}
On the other hand,
according to Theorem 2 in \cite{P1}, there exist homogeneous polynomials
$\psi_1,\cdots,\psi_{n-1}\in k[S_n]$ such that
\begin{equation}
x^{p^n}-\sum\limits_{i=1}^{n-1}\psi_i(x)x^{p^{i}}=0,
\qquad \forall x \in S_n\,.
\end{equation}
Let
\begin{equation}
\wt Q(t,x)=t^{p^n}-\sum\limits_{i=1}^{n-1}\psi_i(x)t^{p^{i}}.
\end{equation}
Since both $Q(x,t)$ and $\wt Q(t,x)$ are minimal $p$-polynomial on $S_n$, we readily obtain:
\begin{prop}[{\cite[Proposition 2.4]{WCL}}]
The polynomial $\v_0$ is zero on $S_n$ and $\v_i=\psi_i$ for $i=1,\cdots,n-1$.
\end{prop}

In \cite{WCL}, the authors investigated the nilpotent elements of $S_n$ and the nilpotent variety $\N(S_n)$ is proved to be an irreducible
complete intersection of codimension $n-1$. Furthermore, by modifying a result in \cite{BFS}, the authors establish an analogue of the Chevalley Restriction Theorem:
$$k[S_n]^{\wt G}\cong k[T]^{{\rm GL}_{n-1}(\mathbb{F}_{p})}\cong k[W_{n-1}]^{G'} ,$$
and $k[S_n]^{\wt G}=k[\v_{1}^{1/p},\cdots,\v_{n-1}^{1/p}]$, where $T$ is the standard generic torus of $S_n$ defined in \cite{BFS} and $G'$ is the restricted automorphism group of $W_{n-1}$.

Similar to section \ref{propertiesofvg}, we define the adjoint quotient $\Phi_{\wt G}$ for the special algebra $S_n$.
Choose generators $\v_{1}^{1/p},\v_{2}^{1/p},\cdots,\v_{n-1}^{1/p}\in k[S_n]^{\wt G}$ for $k[S_n]^{\wt G}$ as a $k$-algebra. Set
$$\Phi_{\wt G}:S_n\longrightarrow \mathbb{A}^{n-1},~~~~~~~~~~~x\mapsto (\v_{1}^{1/p}(x),\v_{2}^{1/p}(x),\cdots,\v_{n-1}^{1/p}(x)).$$
Recall that the following useful morphism:
$$\sigma:~W_{n-1}\longrightarrow S_n,~~~~~~~~~~~~~~~~~x\mapsto x-\div(x)x_nD_n.$$
$\sigma$ is an injective homomorphism of restricted Lie algebras (\cite[Lemma 4.1]{BFS}).

Now we are ready to describe the properties of morphism $\Phi_{\wt G}$.
\begin{lem}\label{commutativeddiagram}
The following diagram is commutative.
\begin{center} \mbox{ }
\xymatrix{
      W_{n-1}\ar@{->}[d]^-{\sigma}  \ar@{->}[r]^-{\Phi_{G'}}
      &   \mathbb{A}^{n-1}  \ar@{->}[d]^-{{\rm Id}}\\
       ~~~~S_n \ar@{->}[r]^{\Phi_{\wt G}}
      &    \mathbb{A}^{n-1}  }
\end{center}
In particular, $\Phi_{\wt G}$ is surjective.
\end{lem}
\begin{proof}
As (\ref{characteristicpolynomialofWn}),
we denote by
$$P'(t,x)=t^{p^{n-1}}-\sum_{i=0}^{n-2}\phi^{'}_i(x)t^{p^i}$$
the characteristic polynomial for $W_{n-1}$.
Hence, for any $x\in W_{n-1}$,
\begin{equation}\label{1}
x^{p^{n-1}}-\sum_{i=0}^{n-2}\phi^{'}_i(x)x^{p^i}=0.
\end{equation}
On the other hand, $Q(t,x)$ is the characteristic polynomial of $S_n$. It follows that
\begin{equation}\label{2}
(\sigma(x))^{p^n}-\sum_{i=1}^{n-1}\v_i(\sigma(x))(\sigma(x))^{p^i}=((\sigma(x))^{p^{n-1}}-\sum_{i=1}^{n-1}\v_i^{1/p}(\sigma(x))(\sigma(x))^{p^{i-1}})^p=0
\end{equation}
for any $x\in W_{n-1}$.
Since both $P'(t,x)$ and $Q(t,x)$ are $p$-polynomial, and therefore all its roots have same multiplicity which is a power of $p$, it follows by
a comparison of (\ref{1}) and (\ref{2}) that $\phi^{'}_i(x)=\v_{i+1}^{1/p}(\sigma(x))$ for any $x\in S_n$. Thus, $\phi^{'}_i=\v_{i+1}^{1/p}\circ\sigma$.
The surjectivity of $\Phi_{G'}$ implies that $\Phi_{\wt G}$ is surjective.
\end{proof}
We still denote by $P_{\e}=\Phi_{\wt G}^{-1}(\e)$ the fiber of $\Phi_{\wt G}$,
where $\e=(\e_{1},\cdots,\e_{n-1})\in\mathbb{A}^{n-1}$.

Let
$$\D=D_1+x_1^{p-1}D_2+\cdots+x_1^{p-1}\cdots x_{n-2}^{p-1}D_{n-1}\,.$$
Considering $\D$ as a linear transformation
on $B_n$, we can easily verify that

$$\ker \D=K \triangleq \left\{a_0+a_1x_n+\cdots+a_{p-1}x_n^{p-1}\,\big|\,
a_0,a_1,\cdots,a_{p-1} \in k\right\} \subseteq B_n;$$
$$
\im \D =k\text{-subspace spanned by~} x_1^{a_1}\cdots x_n^{a_n}
\text{~with~} (a_1,\cdots,a_{n-1}) \neq (p-1,\cdots,p-1).
$$
Take $\e=(\e_{1},\cdots,\e_{n-1})\in\mathbb{A}^{n-1}$, let
$$\D_{\e}=\D+x_1^{p-1}\cdots x_{n-1}^{p-1}\sum\limits_{i=1}^{n-1}(-1)^{n-i}\e_iD_i\in W_{n-1}.$$
Note that, for $\e=0=(0,\cdots,0)\in\mathbb{A}^{n-1}$, $\D_{0}=\D$.

For arbitrary $f_1,\cdots,f_{n-1} \in K $,
denote by

$$\Omega_{f_1,f_2,\cdots,f_{n-1}}^{\e}=\sigma(\D_{\e})+x_1^{p-1}\cdots x_{n-1}^{p-1}\big(D_n(f_1)D_1+\cdots+D_n(f_{n-1})D_{n-1}\big)  \label{Omegaf_i}$$
$$+x_1^{p-2}\cdots x_{n-1}^{p-2}\big(f_1x_2\cdots x_{n-1}+
x_1f_2x_3\cdots x_{n-1}+\cdots+x_1x_2\cdots x_{n-2}f_{n-1}\big)D_n.$$

Set
\begin{eqnarray}
&\Omega^{\e}=&\Big\{\Omega_{f_1,f_2,\cdots,f_{n-1}}^{\e}
\,\big|\, f_i \in K, \deg f_i \geq 2, \forall 1\leq i\leq n-1\Big\}.\label{Omega}
\end{eqnarray}
It is easy to see that $\Omega^{\e}$ is an irreducible subset of $S_n$ and $\dim_k \Omega^{\e}=(n-1)(p-2)$.

\begin{lem}\label{PhiG'}
$\Phi_{G'}(\D_{\e})=\e$.
\end{lem}
\begin{proof}
Let $\D'=x_1^{p-1}\cdots x_{n-1}^{p-1}\sum\limits_{i=1}^{n-1}(-1)^{n-i}\e_iD_i$, $\D_{\e}=\D+\D'$. Note that $\D'$ in the highest graded space.
By Jacobson's formula
$$\D_{\e}^{p^i}\equiv(-1)^{i}D_{i+1}~~ ({\rm mod} \mathfrak{m}W_{n-1}).$$
for all $i=0,\cdots,n-2$.
Since $W_{n-1}$ is a free module $B_{n-1}$-module with a basis $D_1,\cdots,D_{n-1}$.
The above argument implies $\D_{\e}^{p^i}$ with $i=0,1,\cdots n-2,$ form a basis for $W_{n-1}$ over $B_{n-1}$.
Hence we can write
$$\D_{\e}^{p^{n-1}}=\sum_{i=0}^{n-2}f_i\D_{\e}^{p^i} {\rm with}~~f_i\in B_{n-1}.$$
Since $\D_{\e}$ centralizers all power of $\D_{\e}$, we have
$$\sum_{i=0}^{n-2}\D_{\e}(f_i)\D_{\e}^{p^i}=0.$$
Note that $\D_{\e}^{p^i}$ is a basis for $W_{n-1}$ over $B_{n-1}$.
Hence $f_{0},\cdots,f_{n-2}$ are annihilated by all derivations of $B_{n-1}$, this implies that $f_{i}\in k$.
Furthermore, use the same induction argument in \cite[Lemma 4.1,4.2]{WCL}, $$\D_{\e}^{p^{n-1}}\equiv (-1)^{n-1}\sum\limits_{i=1}^{n-1}(-1)^{n-i}\e_iD_i~~({\rm mod} \mathfrak{m}W_{n-1}).$$
Hence, $f_i=\e_{i+1}$ for all $i=0,\cdots,n-2$.
Furthermore, observe that the $B_{n-1}^{\D_{\e}}=k$, this implies that the minimal polynomial of $\D_{\e}$ is of degree $p^{n-1}$.
Hence, the characteristic polynomial is just the minimal polynomial, so $\phi_{i}^{'}(\D_{\e})=f_{i}=\e_{i+1}$ for all $i=0,\cdots,n-2$, as desired.
\end{proof}
\begin{thm}\label{fiberbukeyue}
The morphism $\Phi_{\wt G}$ is equidimensional. For every $\e\in\mathbb{A}^{n-1}$ the fiber $P_{\e}$ is irreducible of codimension $n-1$ in $S_n$,
and contains an affine subspace $\Omega^{\e}$ such that $P_{\e}=\overline{\wt G.\Omega^{\e}}$. The ideal
$I(P_{\e})=\{f\in k[S_n]~\mid~f(P_{\e})=0 \}$ is generated by polynomial $\v_{1}^{1/p}-\e_{1},\v_{2}^{1/p}-\e_2,\cdots,\v_{n-1}^{1/p}-\e_{n-1}$.
In particular, $P_{\e}$ is a complete intersection.
\end{thm}
\begin{proof}
The proof of irreducibility and complete intersection can be applied the same argument as in \cite[Lemma 13]{Pr-1}.

Now we want to prove that $\Omega^{\e}\subseteq P_{\e}$. Let $x=\Omega_{f_1,f_2,\cdots,f_{n-1}}^{\e}\in\Omega^{\e}$.
For any $0\neq a\in k$, let $\tau\in \Aut(B_n)$ be defined by
$$\tau(x_1)=x_1,\cdots,\tau(x_{n-1})=x_{n-1},\tau(x_n)=ax_n,$$
and let $g\in G$ be the corresponding automorphism of $W_n$. According to (\ref{aut1}), $g\in \wt G=\Aut(S_n)$.
By taking limit as $a$ goes to $0$, we have $\sigma(\D_{\e})\in \overline{\wt G.x}$. It follows from Lemma \ref{commutativeddiagram} and
Lemma \ref{PhiG'} that
$$\Phi_{\wt G}(x)=\Phi_{\wt G}(\sigma(\D_{\e}))=\Phi_{G'}(\D_{\e})=\e.$$
This implies the following well-defined morphism,
$$\mu: \wt G\times\Omega^{\e}\rightarrow P_{\e},~~~~(g,x)\mapsto g(x).$$
According to \cite[Remark 5.3]{WCL}, the dimension of the fiber of $\mu$ is just $(n-1)(p-1)+1$.
Hence, $$\dim \im(\mu)=\dim \wt G+\dim \Omega^{\e}-(n-1)(p-1)+1=N-(n-1).$$
Thus $P_{\e}=\overline{\wt G.\Omega^{\e}}$.
\end{proof}
\begin{rem}
In section \ref{propertiesofvg}, the adjoint quotient map $\Phi_{G}$ for Jacobson-Witt algebra $W_n$ is flat morphism since $k[\g]$ is a
free module over $k[\g]^G$ (\cite[Corollary 3]{Pr-1}). In the case of special algebra $S_n$, the adjoint quotient map $\Phi_{\wt G}$ is still a flat morphism (cf.\cite{N}).
\end{rem}

\subsection{transverse slices}
In the classical setting \cite{Ko}, Kontant defined the transversal plane to a principal nilpotent element (see also \cite{Jan} for details).
In this section, we want to construct the analogs for $W_n$ and $S_n$ (see \cite{Sk} for the case of Possion algebra).
Recall that
$$\D_{\e}=\D+x_1^{p-1}\cdots x_{n-1}^{p-1}\sum\limits_{i=1}^{n-1}(-1)^{n-i}\e_iD_i.$$
Let
$$\Omega_{1}^{'}=\{\D_{\e}\mid \e\in\mathbb{A}^{n-1}\}\subseteq W_{n-1}$$
and
$$\Omega_1=\sigma(\Omega_{1}^{'})=\{\sigma(\D_{\e})\mid \e\in\mathbb{A}^{n-1}\}\subseteq S_n.$$

Note that Lemma \ref{commutativeddiagram} and Lemma \ref{PhiG'} implies $\Phi_{\wt G}(\sigma(\D_{\e}))=\Phi_{G'}(\D_{\e})=\e$.
The following proposition is a stronger version of \cite[Lemma 5.5]{WCL}.
\begin{prop}\label{transverseslicetoSn}
The restriction map $k[S_n]^{\wt G}\rightarrow k[\Omega_1]$ is an isomorphism.
\end{prop}
\begin{proof}
We only need to prove the surjectivity, since the map is injective (\cite[Lemma 5.5]{WCL}).
$\Omega_1$ is an affine subspace of $S_n$,
so $k[\Omega_1]$ is a polynomial ring in the $n-1$ coordinate functions $\D_{\e}\mapsto (-1)^{n-i}\e_i$. But
$\v_i^{1/p}|_{\Omega_1}$ are nonzero scalar multiples of these coordinate functions. It follows from
$k[S_n]^{\wt G}\cong k[\v_{1}^{1/p},\cdots,\v_{n-1}^{1/p}]$ that the restriction is surjective.
\end{proof}
Note that the isomorphism $k[S_n]^{\wt G}\cong k[W_{n-1}]^{G'}$ induced by $\sigma$, as a corollary, we readily obtain:
\begin{cor}
The restriction map $k[W_{n-1}]^{G'}\rightarrow k[\Omega_1^{'}]$ is an isomorphism.
\end{cor}
Recall that the adjoint quotient map for $W_{n-1}$:
$$\Phi_{G'}:W_{n-1}\longrightarrow \mathbb{A}^{n-1},~~~~~~~~~~~x\mapsto (\phi_{0}'(x),\phi_{1}'(x),\cdots,\phi_{n-2}'(x)).$$
It follows from Lemma \ref{PhiG'} that $\Phi_{G'}|_{\Omega_{1}^{'}}$ is an isomorphism of varieties $\Omega_{1}^{'}\overset{\sim}{\rightarrow}\mathbb{A}^{n-1}$. Therefore its tangent map is an isomorphism $T_{x}(\Omega_{1}^{'})\rightarrow\mathbb{A}^{n-1}$ for all $x\in \Omega_{1}^{'}$.
Note that $\Phi_{G'}$ is constant on each $G'$-orbit $\O_x$ with $x\in W_{n-1}$ since $\phi_{i}'\in k[W_{n-1}]^{G'}$.
It follows that the tangent map $(d\Phi_{G'})_x$ is zero on the tangent space to $\O_x$:
\begin{equation}\label{constantonorbit}
(d\Phi_{G'})_x(T_x\O_x)=0,
\end{equation}
and
\begin{equation}\label{jiaoshiyigedian}
\Omega_{1}^{'}\cap\O_x=\{x\}\text{\quad for~~all ~~~~~}x\in\Omega_{1}^{'}.
\end{equation}
Furthermore, it follows from (\ref{constantonorbit}) that the isomorphism $T_{x}(\Omega_{1}^{'})\rightarrow\mathbb{A}^{n-1}$ induced by $(d\Phi_{G'})_x$ implies
\begin{equation}\label{qiekongjianjiaoshi0}
T_{x}(\Omega_{1}^{'})\cap T_x(\O_x)=\{0\} \text{\quad for~~all ~~~~~}x\in\Omega_{1}^{'}.
\end{equation}

In fact, similar to \cite[lemma 12]{Pr-1}, we have the stabilizer of $\D_{\e}$ in $G'$ is trivial.
\begin{lem}
We have $\dim T_x(\O_x)=\dim G'$ for all $x\in \Omega_{1}^{'}$.
\end{lem}
Note that $\dim W_{n-1}=\dim G'+(n-1)$ since $\Lie(G')=(W_{n-1})_0$.
Hence the above lemma and (\ref{qiekongjianjiaoshi0}) implies that
\begin{equation}\label{zhiheshiWn-1}
W_{n-1}=T_x(\O_x)\oplus T_x(\Omega_{1}^{'})\text{\quad for~~all ~~~~~}x\in\Omega_{1}^{'}.
\end{equation}
Consider the morphism
\begin{equation}\label{smoothmorphism}
\theta:G'\times \Omega_{1}^{'}\rightarrow W_{n-1}, \quad (g,x)\mapsto g.x
\end{equation}
\begin{prop}
All tangent maps $(d\theta)_{(g,x)}$ with $(g,x)\in G'\times \Omega_{1}^{'}$ are surjective. In particular, $\theta$ is a smooth morphism.
\end{prop}
\begin{proof}
The tangent map
$$(d\theta)_{(1,x)}:T_{1}(G')\times T_x(\Omega_{1}^{'})\rightarrow T_x(W_{n-1})=W_{n-1}.$$
Consider the composition of
$\theta$ with $$G'\rightarrow G'\times \Omega_{1}^{'},h\mapsto (h,x),$$
we have
$(d\theta)_{(1,x)}$ maps $(y,0)$ to $[y,x]$ for any $y\in T_1(G')=(W_{n-1})_0$. Furthermore, consider the composition $\theta$ with
$$\Omega_{1}^{'}\rightarrow G'\times \Omega_{1}^{'},z\mapsto (1,z),$$
we have $(d\theta)_{(1,x)}$ maps $(0,y')$ to $y'$ for any $y'\in T_x(\Omega_{1}^{'})$.
Therefore $(d\theta)_{(1,x)}$ has image $[(W_{n-1})_0,x]+T_x(\Omega_{1}^{'})\subset T_x(\O_x)+T_x(\Omega_{1}^{'})$.
It follow from the proof of Lemma \ref{PhiG'} that $$\ker\ad(x)\cap (W_{n-1})_0=\{0\}.$$ Hence
$$\im (d\theta)_{(1,x)}=T_x(\O_x)+T_x(\Omega_{1}^{'}).$$
So (\ref{zhiheshiWn-1}) implies that $(d\theta)_{(1,x)}$ is surjective for all $x\in \Omega_{1}^{'}$.

For arbitrary $(g,x)\in G'\times \Omega_{1}^{'}$, we just use the same argument in \cite[7.8]{Jan}.
\end{proof}
\begin{rem}
1).~The above result is very similar to the transverse slice in the classical setting (\cite[Chapter III]{Sl}).\\
2).~A differential criterion for dominant ensures that $\theta$ is dominant, i.e., $\overline{G'.\Omega_{1}^{'}}=W_{n-1}$.
In fact, since the stabilizer of $\D_{\e}\in \Omega_{1}^{'}$ in $G'$ is trivial, by a comparison of dimensions, the map $\theta$ is dominant.\\
3).~Proposition \ref{transverseslicetoSn} does not implies that $\overline{\wt G.\Omega_{1}}=S_n$, in fact, the set $\Omega_{1}$ is "too small",
$\overline{\wt G.\Omega_{1}}\neq S_n$ (\cite[Proposition 5.2]{WCL}).
\end{rem}

\subsection{nilpotent elements in $S_n$}
In this section, we want to describe the nilpotent elements in $S_n$.
Recall that  $\sigma: W_{n-1}\rightarrow S_n$ is the homomorphism of restricted Lie algebras defined in section \ref{propertiesofvgforSn}.
\begin{lem}\label{bukongyiweizhebukong}
Let $\wt G$ be the restricted automorphism group of $S_n$ and $x\in S_n$. If $\overline{\wt G.x}\cap \sigma(W_{n-1})\neq \emptyset$, then
$\overline{\wt G.x}\cap T_n\neq \emptyset$, where $$T_n=\langle x_1D_1-x_nD_n,\cdots,x_{n-1}D_{n-1}-x_nD_n\rangle.$$
\end{lem}
\begin{proof}
Take $y\in\overline{\wt G.x}\cap \sigma(W_{n-1})$ and denote by $y'=\sigma^{-1}(y)\in W_{n-1}$.
Let $$
T'_{n-1}=\langle x_1D_1,\cdots,x_{n-1}D_{n-1}\rangle
$$ be the standard maximal torus of $W_{n-1}$.
Thanks to \cite[Lemma 14]{Pr-1}, $\overline{G'.y'}\cap T'_{n-1}\neq \emptyset$.
Let $z'\in\overline{G'.y'}\cap T'_{n-1}$. Note that $\sigma(T'_{n-1})=T_n$.
Hence
$$\sigma(z')\in\sigma(\overline{G'.y'}\cap T'_{n-1})\subset\overline{\wt G.\sigma(y')}\cap \sigma(T'_{n-1})=\overline{\wt G.y}\cap T_{n}\subset \overline{\wt G.x}\cap T_n.$$
\end{proof}
\begin{lem}\label{bukong}
Keep the notations as above. Let $x=\sum\limits_{i=1}^{n}f_iD_i\in S_n$. If $x_{i_0}|f_{i_0}$ for some $i_0\in\{1,\cdots,n\}$, then
$\overline{\wt G.x}\cap T_n\neq \emptyset$.
\end{lem}
\begin{proof}
Let $\tau_{i_0}$ be an automorphism of $B_n$ defined by:
$$\tau_{i_0}(x_j)=x_j \text{~if~} j\neq i_0 \text{~or~} n,\tau_{i_0}(x_{i_0})=x_n,\tau_{i_0}(x_{n})=x_{i_0}.$$
Let $g_{i_0}$ be the corresponding automorphism of $S_n$. Write $g_{i_0}.x=\sum\limits_{i=1}^nf'_iD_i$, it is easy to verify that $x_n|f'_n$.
Without loss of generality, we can assume $i_0=n$, i.e., $x_n|f_n$.
$f_n(x_1,\cdots,x_n)$ can be uniquely written as
$$f_n(x_1,\cdots,x_n)=\sum_{j\geq 1}f_{n_j}(x_1,\cdots,x_{n-1})x_n^j.$$
For any $0\neq c\in k$, Define $\tau_c\in \Aut(B_n)$:
$$\tau_c(x_1)=x_1,\cdots,\tau_c(x_{n-1})=x_{n-1},\tau_c(x_n)=c^{-1}x_n.$$
Let $g_c$ be the corresponding automorphism, then
$$g_c(x)=\sum\limits_{i=1}^{n-1}f_i(x_1,\cdots,x_{n-1},cx_n)D_i+\sum_{j\geq 1}f_{n_j}(x_1,\cdots,x_{n-1})c^{j-1}x_n^jD_n.$$
By taking limit as $c$ goes to $0$, it follows that the element
$$g_0(x):=\sum\limits_{i=1}^{n-1}f_i(x_1,\cdots,x_{n-1},0)D_i+f_{n_1}(x_1,\cdots,x_{n-1})x_nD_n$$
belongs to the closure of $\wt G$-orbit of $x$.
Since $g_0(x)\in S_n$, we have $\div(g_0(x))=0$. Therefore,
$$g_0(x)=\D_1-\div(\D_1)x_nD_n=\sigma(\D_1),$$
where $\D_1=\sum\limits_{i=1}^{n-1}f_i(x_1,\cdots,x_{n-1},0)D_i\in W_{n-1}$.
Hence, $\sigma(\D_1)\in\overline{\wt G.x}\cap \sigma(W_{n-1}).$
Lemma \ref{bukongyiweizhebukong} implies $\overline{\wt G.x}\cap T_n\neq \emptyset$.
\end{proof}
Define $\Omega$ as follows,
$$\Omega=\bigcup\limits_{\e\in\mathbb A^{n-1}}\Omega^{\e},$$
where $\Omega^{\e}$ defined in (\ref{Omega}).
In \cite[Proposition 5.2]{WCL}, the authors proved $\overline{\wt G.\Omega}=S_n$
and $\overline{\wt G.\Omega^{0}}=\N(S_n)$ (see also Theorem \ref{fiberbukeyue}).
The following result is a direct consequence of Lemma \ref{bukong}.
\begin{cor}
There exists a dense subset $\Delta$ of $S_n$ such that $\overline{\wt G.x}\cap T_n\neq\emptyset$ for any $x\in\Delta$.
Furthermore, there exists a dense subset $\Lambda$ of nilpotent variety $\N(S_n)$ such that $0\in \overline{\wt G.x}$ for any $x\in\Lambda$.
\end{cor}
\begin{rem}
Premet proved that $x\in W_n$ is nilpotent if and only if $0\in\overline{G.x}$ (\cite[Corollary 6]{Pr-1}).
It is clear $x\in S_n$ is nilpotent if $0\in \overline{\wt G.x}$. The above Corollary give a partial inverse problem. We conjecture that $\overline{\wt G.x}\cap\sigma(W_{n-1})\neq\emptyset$ for any $x\in S_n$.
\end{rem}

\end{document}